\documentclass[11pt,a4paper]{amsart}

\usepackage{amsmath,amsfonts,amsthm}
\usepackage{graphicx}

\newcommand{\N}{\mathbb {N}}
\newcommand{\Z}{\mathbb {Z}}

\DeclareMathOperator{\closure}{closure}
\DeclareMathOperator{\diameter}{diamater}

\newtheorem{theorem}{Theorem}
\newtheorem{lemma}{Lemma}
\newtheorem{remark}{Remark}

\newtheorem{proposition}{Proposition}

\begin{document}

\title[Dimension and measure of baker-like skew-products]{Dimension and measure of baker-like skew-products of $\boldsymbol{\beta}$-transformations}
\author{David F\"arm}
\address{David F\"arm, Institute of Mathematics, Polish Academy of
  Sciences ulica \'Sniadeckich 8, P.O. Box 21, 00-956 Warszawa,
  Poland}
\curraddr{Centre for Mathematical Sciences, Box~118, 22~100~Lund, Sweden}
\email{david@maths.lth.se}

\author{Tomas Persson}
\address{Tomas Persson, Institute of Mathematics, Polish Academy of
  Sciences ulica \'Sniadeckich 8, P.O. Box 21, 00-956 Warszawa,
  Poland}
\curraddr{Centre for Mathematical Sciences, Box~118, 22~100~Lund, Sweden}
\email{tomasp@maths.lth.se}

\thanks{Both authors were supported by
EC FP6 Marie Curie ToK programme CODY. Part of the paper was written
when the authors were visiting institut Mittag-Leffler in
Djursholm. The authors are grateful for the hospitality of the
institute. The authors would like to thank Lingmin Liao for pointing
out the articles \cite{BrownYin} and \cite{Kwon}.}

\begin{abstract}
  We consider a generalisation of the baker's transformation,
  consisting of a skew-product of contractions and a
  $\beta$-transformation. The Hausdorff dimension and Lebesgue measure
  of the attractor is calculated for a set of parameters with positive
  measure. The proofs use a new transverality lemma similar to
  Solomyak's \cite{Solomyak}. This transversality, which is applicable
  to the considered class of maps holds for a larger set of parameters
  than Solomyak's transversality.
\end{abstract}


\maketitle

\section{Introduction}\label{5sec:introduction}

In \cite{AlexanderYorke}, Alexander and Yorke considered fat baker's
transformations. These are maps on the square $[0, 1) \times [0, 1)$,
    defined by
\[
  (x, y) \mapsto \left\{ \begin{array}{ll} (\lambda x,
    2 y) & \text{if } y < 1 / 2 \\ (\lambda x + 1 - \lambda,
    2 y - 1) & \text{if } y \geq 1 / 2 \end{array} \right. ,
\]
where $\frac{1}{2} < \lambda < 1$ is a parameter, see Figure
\ref{5fig:beta20lambda06}. They showed that the \textsc{srb}-measure
of this map is the product of Lebesgue-measure and (a rescaled version
of) the distribution of the corresponding Bernoulli convolution
\[
\sum_{k = 1}^\infty \pm \lambda^k.
\]
Together with Erd\H os' result \cite{Erdos}, this implies that if
$\lambda$ is the inverse of a Pisot-number, then the
\textsc{srb}-measure is singular with respect to the Lebesgue measure
on $[0, 1) \times [0, 1)$.

\begin{figure}
  \begin{center}
      \includegraphics[width=0.6\textwidth]{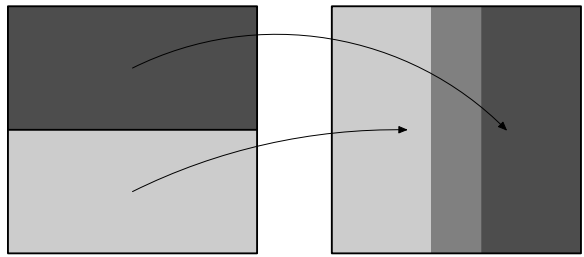}
  \end{center}
  \caption{The fat baker's transformation for $\lambda = 0.6$.}
  \label{5fig:beta20lambda06}
\end{figure}

In \cite{Solomyak}, Solomyak proved that for almost all $\lambda \in
(\frac{1}{2}, 1)$, the distribution of the corresponding Bernoulli
convolution $\sum_{k = 1}^\infty \pm \lambda^k$ is absolutely
continuous with respect to Lebesgue measure. Hence this implies that the
\textsc{srb}-measure of the fat baker's transformation is absolutely
continuous for almost all $\lambda \in (\frac{1}{2}, 1)$. Solomyak's
proof used a transversality property of power series of the form $g
(x) = 1 + \sum_{k = 1}^\infty a_k x^k$, where $a_k \in \{-1, 0, 1\}$.
More precisely, Solomyak proved that there exists a $\delta > 0$ such
that if $x \in (0, 0.64)$ then
\begin{equation} \label{5eq:solomyakstrans}
  | g (x) | < \delta \quad \Longrightarrow \quad g' (x) < -
  \delta.
\end{equation}
This property ensures that if the graph of $g(x)$ intersects the
$x$-axis it does so at an angle which is bounded away from 0, thereby
the name transversality. The constant $0.64$ is an approximation of a
root to a power series and cannot be improved to something larger than
this root. A simplified version of Solomyak's proof appeared in the paper
\cite{PeresSolomyak}, by Peres and Solomyak. We will make use of the
method from this simpler version.

In this paper we consider maps of the form
\begin{equation}\label{5eq:ourmap}
  (x, y) \mapsto \left\{ \begin{array}{ll} (\lambda x,
    \beta y) & \text{if } y < 1 / \beta \\ (\lambda x + 1 - \lambda,
    \beta y - 1) & \text{if } y \geq 1 / \beta \end{array} \right. ,
\end{equation}
where $0 < \lambda < 1$ and $1 < \beta < 2$, see Figure
\ref{5fig:beta17lambda08}. Using the above mentioned transversality of
Solomyak one can prove that for almost all $\lambda \in (0, 0.64)$ and
$\beta \in (1, 2)$ the \textsc{srb}-measure is absolutely continuous
with respect to Lebesgue measure provided $\lambda \beta > 1$, and the
Hausdorff dimension of the \textsc{srb}-measure is $1 + \frac{\log
  \beta}{\log 1/\lambda}$ provided $\lambda \beta < 1$.

\begin{figure}
  \begin{center}
      \includegraphics[width=0.6\textwidth]{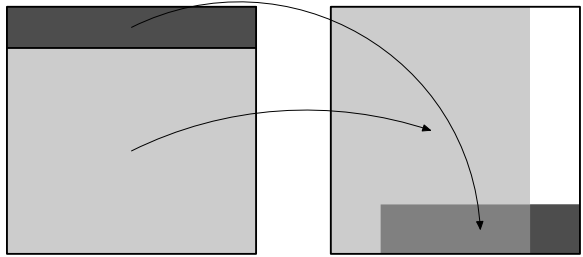}
  \end{center}
  \caption{The map \eqref{5eq:ourmap} for $\beta=1.2$ and $\lambda=0.8$}
  \label{5fig:beta17lambda08}
\end{figure}

A problem with this approach is that the condition $\lambda<0.64$ is
very restrictive when $\beta$ is close to 1. Then the above method
yields no $\lambda$ for which the \textsc{srb}-measure is absolutely
continuous, and it does not give the dimension of the
\textsc{srb}-measure for any $\lambda \in (0.64,1/\beta)$.

We prove that these results about absolute continuity and dimension of
the \textsc{srb}-measure hold for sets of $(\beta, \lambda)$ of
positive Lebesgue measure, even when $\lambda > 0.64$. This is done by
extending the interval on which the transversality property
\eqref{5eq:solomyakstrans} holds. This can be done in our setting,
since in our class of maps, not every sequence $(a_k)_{k = 1}^\infty$
with $a_k \in \{-1, 0, 1\}$ occurs in the power series $g (x) = 1 +
\sum_{k = 1}^\infty a_k x^k$ that we need to consider in the proof.
To control which sequences that occur, we will use some results of
Brown and Yin \cite{BrownYin} and Kwon \cite{Kwon} on natural
extensions of $\beta$-shifts.

The paper is organised as follows. In Section~\ref{5sec:betashifts} we
recall some facts about $\beta$-transformations and $\beta$-shifts. We
then present the results of Brown and Yin, and Kwon in
Section~\ref{5sec:symetric}. In Section~\ref{5sec:results} we state our
results, and give the proofs in Section~\ref{5sec:proofs}. The
transversality property is stated and proved in
Section~\ref{5sec:transversality}.

\section[$\beta$-shifts]{$\beta$-shifts} \label{5sec:betashifts}

Let $\beta > 1$ and define $f_\beta \colon [0, 1] \to [0, 1)$ by
  $f_\beta (x) = \beta x$ modulo 1. For $x \in [0, 1]$ we associate a
  sequence $d (x, \beta) = (d_k (x, \beta) )_{k = 1}^\infty$ defined
  by $d_k (x, \beta) = [ \beta f_\beta^{k - 1} (x) ]$ where $[ x ]$
  denotes the integer part of $x$. If $x \in [0, 1]$, then $x =
  \phi_\beta ( d (x, \beta) )$, where
\[
  \phi_\beta (i_1, i_2, \ldots) = \sum_{k = 1}^\infty
  \frac{i_k}{\beta^k}
\]
This representation, among others, of real numbers was studied by
R\'enyi \cite{Renyi}. He proved that there is a unique probability
measure $\mu_\beta$ on $[0,1]$ invariant under $f_\beta$ and
equivalent to Lebesgue measure. We will use this measure in Section
\ref{5sec:proofs}.

We let $S_\beta^+$ denote the closure in the product topology of the
set $\{\, d (x, \beta) : x \in [0, 1) \,\}$. The compact symbolic
  space $S_\beta^+$ together with the left shift $\sigma$ is called a
  $\beta$-shift. If we define $d_- (1, \beta)$ to be the limit in
    the product topology of $d (x, \beta)$ as $x$ approaches $1$ from the left, we
    have the equality
\begin{multline} \label{5eq:Sbeta+}
  S_\beta^+ = \{\, (a_1, a_2, \ldots) \in \{0, 1, \ldots, [\beta]
  \}^\N : \\ \sigma^k (a_1, a_2, \ldots) \leq d_- (1, \beta) \ \forall k
  \geq 0 \, \},
\end{multline}
where $\sigma$ is the left-shift. This was proved by Parry in
\cite{Parry}, where he studied the $\beta$-shifts and their invariant
measures. Note that $d_- (1, \beta) = d (1,\beta)$ if and only if $d
(1, \beta)$ contains infinitely many non-zero digits. A particularly
useful property of the $\beta$-shift is that $\beta<\beta'$ implies
$S^+_\beta \subset S^+_{\beta'}$. The map $\phi_\beta \colon
S_\beta^{+} \to [0,1]$ is not necessarily injective, but we have
$d(\cdot,\beta)\circ f_\beta=\sigma \circ d(\cdot,\beta)$.

\section{Symmetric ${\beta}$-shifts} \label{5sec:symetric}

Let $\beta > 1$ and consider $S_\beta^+$. The natural extension of
$(S_\beta^+, \sigma)$ can be realised as $(S_\beta, \sigma)$, with
\[
  S_\beta = \{\, (\ldots, a_{-1}, a_0, a_1, \ldots) : (a_n, a_{n+1},
  \ldots) \in S_\beta^+ \ \forall n \in \Z \,\},
\]
where $\sigma$ is the left shift on bi-infinite sequences. We will use
the concept of cylinder sets only in $S_\beta$.  A cylinder set is a
subset of $S_\beta$ of the form
\[
  [a_{-n}, a_{-n+1}, \ldots, a_0] = \{\, (\ldots, b_{-1}, b_0, b_1,
  \ldots) \in S_\beta : a_k = b_k \ \forall k = -n, \ldots, 0 \,\}.
\]
We define $S_\beta^-$ to be the set
\begin{align*}
  S_\beta^- & = \{\, (b_1, b_2, \ldots) : \exists (a_1, a_2, \ldots)
  \in S_\beta^+ \text{ s.t. } (\ldots, b_2, b_1, a_1, a_2, \ldots) \in
  S_\beta \,\} \\
  & = \{\, (b_1, b_2, \ldots) : (\ldots, b_2, b_1, 0, 0, \ldots) \in
  S_\beta \,\}.
\end{align*}
We will be interested in the set $S$ of $\beta$ for which $S_\beta^+ =
S_\beta^-$. This set was considered by Brown and Yin in
\cite{BrownYin}. We now describe the properties of $S$ that we will
use later on.

Consider a sequence of the digits $a$ and $b$. Any such sequence can be
written in the form
\[
  (a^{n_1}, b, a^{n_2}, b, \ldots),
\]
where each $n_k$ is a non-negative integer or $\infty$. We say that
such a sequence is allowable if $a\in \N$, $b = a - 1$, and $n_1 \geq
1$. If the sequence $(n_1, n_2, \ldots)$ is also allowable, we say
that $(a^{n_1}, b, a^{n_2}, b, \ldots)$ is derivable, and we call
$(n_1, n_2, \ldots)$ the derived sequence of $(a^{n_1}, b, a^{n_2}, b,
\ldots)$. For some sequences, this operation can be carried out over
and over again, generating derived sequences out of derived
sequences. We have the following theorem.

\begin{theorem}[Brown--Yin \cite{BrownYin}, Kwon \cite{Kwon}] \label{5the:BrownYinKwon}
  $\beta \in S$ if and only if $d (1, \beta)$ is derivable infinitely
  many times.
\end{theorem}

The ``only if''-part was proved by Brown and Yin in \cite{BrownYin}
and the ``if''-part was proved by Kwon in \cite{Kwon}. Using this
characterisation of $S$, Brown and Yin proved that $S$ has the
cardinality of the continuum, but its Hausdorff dimension is zero.

There is a connection between numbers in $S$ and Sturmian
sequences. We will not make any use of the connection in this paper,
but refer the interested reader to Kwon's paper \cite{Kwon} for details.

For our main results in the next section, it is nice to know whether
$S$ contains numbers arbirarily close to $1$. The following
proposition is easily proved using Theorem~\ref{5the:BrownYinKwon}.

\begin{proposition} \label{5pro:inf}
  $\inf S = 1$.
\end{proposition}

\begin{proof}
  We prove this statement by explicitely choosing sequences $d (1,
  \beta)$ corresponding to numbers $\beta\in S$ arbitrarily close to $1$.
  We do this by first finding some sequences that are infinitely
  derivable, and then we find the corresponding $\beta$ by solving the
  equation $1 = \phi_\beta (d (1, \beta))$. Let us first remark that
  the sequence $(1, 0, 0, \ldots)$ is its own derived sequence. 
  
  The sequence $d (1, \beta) = (1, 1, 0, (1, 0)^\infty )$ is clearly
  derivable infinitely many times. It's derived sequence is $(2, 1, 1,
  \ldots)$, and the derived sequence of this sequence is $(1, 0, 0,
  \ldots)$.  One finds numerically that the corresponding $\beta$ is
  given by $\beta = 1.801938\ldots$ and that $1 / \beta =
  0.554958\ldots$

  There are however smaller numbers in the set $S$. Consider the
  sequence $d (1, \beta) = (1, 0, (1, 0, 0)^\infty )$. It's derived
  sequence is $(1, 1, 0, (1, 0)^\infty) )$, which derives to $(2, 1,
  1, \ldots)$, and so on. Solving for $\beta$ we find that $\beta =
  1.558980\ldots$ and $1 / \beta = 0.641445\ldots$ Now, for all natural $n$, let
  $\beta_n$ be such that
  \[
    d (1, \beta_n) = (1, 0^n, (1, 0^{n+1})^\infty).
  \]
  Then, for $n \geq 2$, the derived sequence of $d (1, \beta_n)$ is
  the sequence $d (1, \beta_{n-1})$. Hence all sequences $d (1,
  \beta_n)$ are infinitely derivable, and so $\beta_n \in S$. Moreover
  it is clear that $\beta_n \to 1$ as $n \to \infty$. See
  Table~\ref{5tab:S}.
\end{proof}

\begin{table}
  \begin{center}
    \begin{tabular}{r|l|l}
      $n$ & $\beta_n$ & $1 / \beta_n$ \\
      \hline
      1 & 1.558980\ldots & 0.641445\ldots \\
      2 & 1.438417\ldots & 0.695209\ldots \\
      3 & 1.365039\ldots & 0.732580\ldots \\
      4 & 1.315114\ldots & 0.760390\ldots \\
      5 & 1.278665\ldots & 0.782066\ldots \\
    \end{tabular}
  \end{center}
  \caption{Some numerical values of $\beta_n$.}
  \label{5tab:S}
\end{table}

\section{Results} \label{5sec:results}

Let $0 < \lambda < 1$ and $1<\beta < 2$. Put $Q = [0, 1) \times [0,
    1)$ and define $T_{\beta,\lambda} \colon Q \to Q$ by
\[
  T_{\beta, \lambda} (x, y) = \left\{ \begin{array}{ll} (\lambda x,
    \beta y) & \text{if } y < 1 / \beta \\ (\lambda x + 1 - \lambda,
    \beta y - 1) & \text{if } y \geq 1 / \beta \end{array} \right. .
\]
Denote by $\nu$ the 2-dimensional Lebesgue measure on $Q$. For any $n
\in \N$ we define the measure
\[
  \nu_n = \frac{1}{n} \sum_{k = 0}^{n-1} \nu \circ T_{\beta,
    \lambda}^{-n}.
\]
The \textsc{srb}-measure (it is unique as noted below) of $T_{\beta,
  \lambda}$ is the weak limit of $\nu_n$ as $n \to \infty$.

The \textsc{srb}-measures are characterised by the property that their
conditional measures along unstable manifolds are equivalent to
Lebesgue measure. The existence of such measures was established for
invertible maps by Pesin \cite{Pesin} and extended to non-invertible
maps by Schmeling and Troubetzkoy \cite{Schmeling}. We denote the
\textsc{srb}-measure of $T_{\beta, \lambda}$ by
$\mu_\textsc{srb}$. Using the Hopf-argument used by Sataev in
\cite{Sataev} one proves that the \textsc{srb}-measure is
unique. (Sataev's paper is about a somewhat different map, but the
argument goes through without changes.) 

The support of $\mu_\textsc{srb}$ is the set
\[
  \Lambda = \closure \bigcap_{n = 0}^\infty T_{\beta, \lambda}^n (Q)
\]
of which we have examples in Figure \ref{5fig:attractorbeta12lambda08}.
\begin{figure}
  \hspace{\stretch{1}}
  \includegraphics[width=0.35\textwidth]{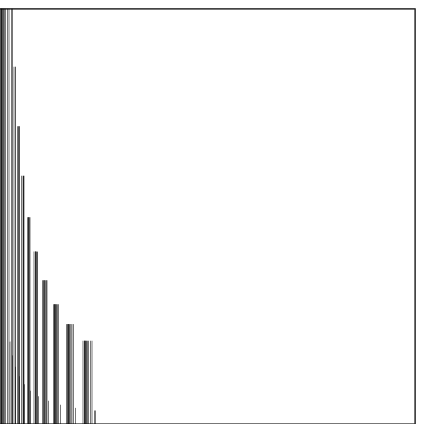} \hspace{\stretch{1}}
  \includegraphics[width=0.35\textwidth]{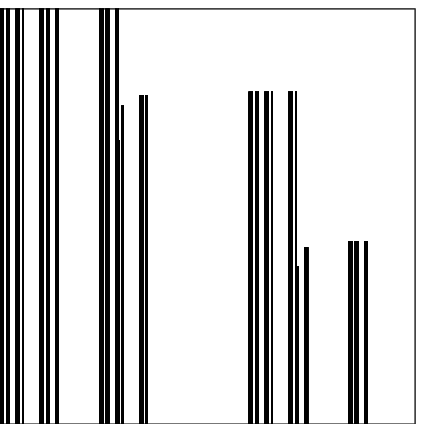} \hspace{\stretch{1}}
  \caption{The set $\Lambda$ for $\beta=1.2$ and $\lambda=0.8$ (left) and $\beta=1.8$ and $\lambda=0.4$ (right).}
  \label{5fig:attractorbeta12lambda08}
\end{figure}
One can estimate the dimension from above by covering the set
$\Lambda$ with the natural covers, consisting of the pieces of
$T_{\beta, \lambda}^n (Q)$. This gives us the upper bound, that the
Hausdorff dimension of $\Lambda$ is at most $1 + \frac{\log
  \beta}{\log 1/\lambda}$. If $\lambda \beta > 1$ this is a trivial
estimate, since then $1 + \frac{\log \beta}{\log 1/\lambda} > 2$. 

The following theorem states that in the case when $\lambda \beta<1$, there is a set of parameters of positive Lebesgue measure for which the estimate above is optimal.

\begin{theorem} \label{5the:dimension}
  Let $1 < \beta < 2$ and $\gamma=\inf\{\, \beta'\in S : \beta' \geq
  \beta \, \}$. Then for Lebesgue almost every $\lambda\in (0,1/\gamma)$
  the Hausdorff dimension of the \textsc{srb}-measure of $T_{\beta,
    \lambda}$ is $1 + \frac{\log \beta}{\log 1/\lambda}$.
\end{theorem}

Recall from Proposition \ref{5pro:inf} that $\inf S=1$. This implies
that when $\beta$ gets close to 1, Theorem \ref{5the:dimension} gives
the dimension of the \textsc{srb}-measure for a large set of
$\lambda>0.64$, which is not obtainable using Solomyak's
transversality from \cite{Solomyak}, described in the introduction.

In the area-expanding case, when $\lambda\beta>1$, we have the
following theorem.

\begin{theorem}\label{5the:abscont}
  For any $\gamma\in S$, there is an $\varepsilon>0$ such
  that for all $\beta$ with $1/\beta \in [1/ \gamma, 1/\gamma +
    \varepsilon)$, and Lebesgue almost every $\lambda \in (1/\beta,
    1/\gamma + \varepsilon)$ the \textsc{srb}-measure of $T_{\beta,
      \lambda}$ is absolutely continuous with respect to Lebesgue
    measure.
\end{theorem}

Since $\inf S=1$ by Proposition~\ref{5pro:inf}, there are $\beta$
arbitrarily close to 1 for which we have a set of $\lambda$ of
positive Lebesgue measure, where the \textsc{srb}-measure is
absolutely continuous. In particular, this means that for these
parameters, the set $\Lambda$ has positive 2-dimensional Lebesgue
measure.

Let us comment on the relation between Theorem~\ref{5the:abscont} and
the results of Brown and Yin in \cite{BrownYin}. Brown and Yin
considers any $\beta > 1$. In the case $1 < \beta < 2$ their result is
the following. They consider the map
\[
  (x, y) \mapsto \left\{ \begin{array}{ll} (\frac{1}{\beta} x, \beta
  y) & \text{if } y < \frac{1}{\beta}, \\ \rule{0pt}{13pt}
  (\frac{1}{\beta} x + \frac{1}{\beta}, \beta y - 1) & \text{if } y
  \geq \frac{1}{\beta}. \end{array} \right.
\]
Hence their map is similar to ours when $\lambda =
\frac{1}{\beta}$. They proved that the Lebesgue measure restricted to
the set $\Lambda$ is invariant if  $\beta \in
S$.

\section{Transversality} \label{5sec:transversality}

The main results of this paper, Theorem \ref{5the:dimension} and
Theorem \ref{5the:abscont}, only deal with $1<\beta<2$. However, the
arguments in this section work just as well for larger $\beta$, so for
the rest of this section we will be working with a fixed $\beta>1$.

Consider the set of power series of the form
\begin{equation} \label{5eq:powerseries}
  g (x) =1+ \sum_{k = 1}^\infty (a_k - b_k) x^k,
\end{equation}
where $(a_1, a_2, \dots)$ and
$ (b_1, b_2, \dots)$ are sequences in
$S_\beta^+$.

\begin{lemma}\label{5lem:transversality}
  There exist $\varepsilon > 0$ and $\delta > 0$ such that for any
  power series $g$ of the form \eqref{5eq:powerseries}, $x \in [0, 1/
    \beta + \varepsilon]$ and $| g(x) | < \delta$ implies that $g' (x)
  < - \delta$.
\end{lemma}

\begin{proof}
  Let 
  \begin{equation}\label{5eq:epsilon}
  0<\varepsilon < \min \bigg\{\frac{1-1/\beta}{2}, \frac{1}{[ \beta]} \bigg\}
  \end{equation}
   and assume that no such $\delta$ exists. We will
  show that if $\varepsilon$ is too small, then we get a
  contradiction.
   
  By assumption, there is a sequence $g_n$ of power series of the form
  \eqref{5eq:powerseries} and a sequence of numbers $x_n \in [0, 1 /
    \beta+\varepsilon]$, such that $\lim_{n \to \infty} g_n (x_n) = 0$
  and $\liminf_{n \to \infty} g_n' (x_n) \geq 0$. We can take a
  subsequence such that $g_n$ converges term-wise to a series
  \[
  g (x) = 1+ \sum_{k = 1}^\infty (a_k - b_k) x^k
  \]
  with $(a_1, a_2,
  \dots), (b_1, b_2, \dots) \in S_\beta^+$, and such that $x_n$ converges to
  some number $x_0\in [0, 1 / \beta+\varepsilon]$. Clearly, $g(x_0)=0$
  and $g'(x_0) \geq 0$, so looking at \eqref{5eq:powerseries} we note
  that $x_0\neq0$.

  Assume first that $x_0 \in (0,1/\beta]$.  Let $\beta_0 = 1 / x_0
  \geq \beta$. Then $g (x_0) = 0$ and $(a_1, a_2, \dots), (b_1, b_2,
  \dots) \in S_{\beta_0}^+$ implies that
  \begin{equation}\label{5eq:extremalsums}
    \phi_{\beta_0} (a_1, a_2, \dots) - \phi_{\beta_0} (b_1, b_2, \dots)
    = \sum_{k = 1}^\infty \frac{a_k}{\beta_0^k} - \sum_{k = 1}^\infty
    \frac{b_k}{\beta_0^k} = - 1.
  \end{equation}
  Both of the sums in \eqref{5eq:extremalsums} are in $[0, 1]$, since
  they equal $\phi_{\beta_0} (a_1, a_2, \dots)$ and $\phi_{\beta_0}
  (b_1, b_2, \dots)$ respectively. We conclude that
  \[
    \sum_{k = 1}^\infty \frac{a_k}{\beta_0^k} = 0 \quad \text{and}
    \quad \sum_{k = 1}^\infty \frac{b_k}{\beta_0^k} = 1.
  \]
  We must therefore have $(a_1, a_2, \ldots) = (0, 0, \ldots)$, and
  $b_k$ must be nonzero for at least some $k$. From
  \eqref{5eq:powerseries} we then get $g ' (x)= -\sum_{k=1}^{\infty}k
  b_kx^{k-1}<0$ for all $x \in (0, 1/ \beta]$, contradicting the fact
  that $g' (x_0) \geq 0$.

  Assume instead that $x_0 \in (1/\beta, 1/\beta + \varepsilon]$. We
  write
  \begin{equation} \label{5eq:twofunctions}
    g (x) =1+h_1(x)-h_2(x),
  \end{equation}
  where
  \begin{equation}\label{5eq:h1h2}
    h_1 (x) = \sum_{k = 1}^\infty a_k x^k \quad \text{and} \quad
    h_2(x)=\sum_{k = 1}^\infty b_k x^k.
  \end{equation}
  Since $(b_1, b_2, \dots) \in S_\beta^+$, we have $h_2 (1/\beta) \leq
  1$. Moreover, for $x\geq 0$ we have $0 \leq h_2' (x) \leq \sum_{k = 1}^\infty [
    \beta ] k x^{k - 1} = \frac{[\beta] }{(1 - x)^2}$. Therefore we
  have
  \begin{equation}\label{5eq:h2}
    h_2 (x_0) \leq 1 + \int_{1/ \beta}^{1/ \beta + \varepsilon}
    \frac{[\beta]\ }{(1 - x)^2} \mathrm{d}x= 1 + \frac{[\beta]
      \varepsilon}{(1 - 1/ \beta - \varepsilon)(1 - 1/ \beta)}.
  \end{equation}
  Since $g(x_0)=0$ we see from \eqref{5eq:twofunctions} and
  \eqref{5eq:h2} that
  \[
  h_1 (x_0) \leq \frac{[\beta] \varepsilon}{(1 - 1/ \beta -
    \varepsilon)(1 - 1/ \beta)}.
  \] 
  If we have $\frac{[\beta] \varepsilon}{(1 - 1/ \beta -
    \varepsilon)(1 - 1/ \beta)}\geq x_0$, then let $k=0$. Otherwise, let
  $k$ be the largest integer such that $x_0^k >
  \frac{[\beta]\varepsilon}{(1 - 1/ \beta - \varepsilon)(1 - 1/
    \beta)}$. Since $h_1(x)$ is of the form \eqref{5eq:h1h2} and all
  its terms are non-negative we must have $a_i = 0$ for $i \leq
  k$. This implies that
  \begin{equation}\label{5eq:h1prim}
    h_1' (x) \leq \sum_{i = k + 1}^\infty [\beta] i x^{i - 1} \leq
    [\beta]\frac{(k+1) x^k + k x^{k + 1}}{(1 - x)^2} = x^{k+1} [\beta]
    \frac{k + 1 + kx}{x (1 - x)^2}.
  \end{equation}
  By the maximality of $k$, we have $x_0^{k+1} \leq
  \frac{[\beta]\varepsilon}{(1 - 1/ \beta - \varepsilon)(1 - 1/
    \beta)}$, so \eqref{5eq:h1prim} and \eqref{5eq:epsilon} implies
  \begin{equation}\label{5eq:h1primnew}
    h_1' (x_0) \leq \frac{[\beta]^2\varepsilon}{(1 - 1/ \beta -
      \varepsilon)(1 - 1/ \beta)} \frac{k + 1 + k x_0}{x_0 (1 -
      x_0)^2}\leq \frac{[\beta]^2\varepsilon(2k + 1)}{(1 - 1/ \beta -
      \varepsilon)^4x_0}.
  \end{equation}
  To estimate $h_2' (x_0)$ from below, we note that since $h_2(x)$ is
  of the form \eqref{5eq:h1h2}, we must have $h_2''(x)\geq 0$ for all
  $x$. We also have $h_2(x_0)\geq 1$ since $0=g(x_0)=h_1(x_0)-h_2(x_0)$. Since $h_2 (0) = 0$, this implies
  \begin{equation}\label{5eq:h2prim}
    h_2' (x_0) \geq \frac{h_2 (x_0)}{x_0} \geq \frac{1}{x_0}.
  \end{equation}

  Now, if we can choose $\varepsilon$ so small that $g'(x_0)=h_1'
  (x_0) - h_2' (x_0) < 0$, we get a contradiction to the fact that
  $g'(x_0)\geq 0$. By \eqref{5eq:h1primnew} and \eqref{5eq:h2prim} we
  see that it is enough to choose $\varepsilon$ so small that
  \[
     \frac{[\beta]^2\varepsilon(2k + 1)}{(1 - 1/ \beta -
       \varepsilon)^4x_0} - \frac{1}{x_0}<0 \quad \Longleftrightarrow
     \quad \varepsilon < \frac{(1 - 1/ \beta -
       \varepsilon)^4}{[\beta]^2(2k+1)}.
  \]
  So, by \eqref{5eq:epsilon} it is sufficient to choose
  \begin{equation}\label{5eq:epsilonbound}
  \varepsilon < \frac{(1- 1 / \beta)^4}{2^4[\beta]^2(2
    k + 1)}.
  \end{equation}
  To get a bound on $k$ recall that by definition, either $k=0$ or it
  satisfies
  \[
  x_0^k > \frac{[\beta] \varepsilon}{(1 - 1/ \beta - \varepsilon)(1
    - 1/ \beta)}.
  \]
  By \eqref{5eq:epsilon}  we get
  \begin{align*}
    k &< \frac{\log( [\beta] \varepsilon)-
      \log(1-1/\beta-\varepsilon)-\log(1-1/\beta)}{\log(x_0)} \\ &<
    \frac{\log( [\beta] \varepsilon)}{\log(1/\beta+\varepsilon)}\leq
    \frac{\log ([\beta]\varepsilon)}{\log(\frac{1+1/\beta}{2})}.
  \end{align*}
  Inserting this estimate into \eqref{5eq:epsilonbound}, we get the
  sufficient condition
  \begin{equation}\label{5eq:finalestimate}
    \varepsilon < \frac{ \bigl( 1 - 1 / \beta
      \bigr)^4}{2^4[\beta]^2\frac{2 \log( [\beta]\varepsilon)}{\log
        \frac{1 + 1 / \beta}{2} } + 2^4[\beta]^2} \Leftrightarrow
    \frac{2^5[\beta]^2}{\log \frac{1 + 1 / \beta}{2} } \varepsilon \log
    ( [\beta] \varepsilon) + 2^4[\beta]^2\varepsilon < \Bigl( 1 - 1 /
    \beta \Bigr)^4.
  \end{equation}
  But $\varepsilon \log \varepsilon \to 0$ as $\varepsilon$ shrinks to
  $0$, so it is clear that we can find an $\varepsilon>0$ satisfing
  \eqref{5eq:finalestimate}.
\end{proof}

\begin{remark}
  Let us give an explicit formula for which $\varepsilon$ we can
  choose in the case $1<\beta<2$. For such $\beta$ we have $[ \beta ]
  =1$. By \eqref{5eq:epsilon} we have $\varepsilon\leq
  \frac{1-1/\beta}{2}$, so it follows that $\varepsilon \leq \frac{-
    \varepsilon \log \varepsilon}{\log \frac{2}{1 - 1 / \beta}}$. This
  implies that \eqref{5eq:finalestimate} is satisfied if
  \[
    -\varepsilon \log \varepsilon \Bigl( \frac{2^5}{\log \frac{2}{1 + 1
        / \beta}} + \frac{2^4}{\log \frac{2}{1 - 1 / \beta}} \Bigr) <
    \bigl( 1 - 1 / \beta \bigr)^4.
  \]
  Finally we use that $- \varepsilon \log \varepsilon < \frac{3}{4}
  \sqrt{\varepsilon}$ and conclude that it is sufficient to pick any
  \[
    \varepsilon \leq \frac{16}{9} \frac{ \bigl( 1 - 1 / \beta
      \bigr)^8}{\Bigl( \frac{2^5}{\log \frac{2}{1 + 1 / \beta}} +
      \frac{2^4}{\log \frac{2}{1 - 1 / \beta}} \Bigr)^2}.
  \]
\end{remark}

\section{Proofs} \label{5sec:proofs}

Before we give the proofs of Theorems~\ref{5the:dimension} and
\ref{5the:abscont}, we make some preparations that will be used in both
proofs.

For fixed $1<\beta<2$ and $0<\lambda< 1$, the set $\Lambda$ satisfies
\begin{equation}\label{5eq:lambda}
  \Lambda = \{\, (x,y) : \exists \boldsymbol{a} \in S_\beta \text{
    such that } x = \pi_1 (\boldsymbol{a}, \lambda),\ y = \pi_2
  (\boldsymbol{a}, \beta) \,\},
\end{equation}
where
\begin{align*}
  \pi_1 (\boldsymbol{a}, \lambda) &= (1 - \lambda) \sum_{k =
    0}^{\infty} a_{-k} \lambda^k, \\
     \pi_2 (\boldsymbol{a}, \beta) &=
  \sum_{k = 1}^\infty a_k \beta^{-k}. 
\end{align*}
To see this one can argue as follows. Recall that $\Lambda$ is the
closure of the set $\bigcap_{n=0}^\infty T_{\beta, \lambda}^n (Q)$.
For each $(x,y) \in \bigcap_{n=0}^\infty T_{\beta, \lambda}^n (Q)$, we
have that $(x,y) = T_{\beta,\lambda}^n (x_n, y_n)$ for some sequence
$(x_n, y_n) \in Q$ with $T_{\beta, \lambda} (x_{n+1},y_{n+1}) = (x_n,
y_n)$. This means that there is a sequence $\boldsymbol{a}\in S_\beta$ such that

\[
 (x,y) = T_{\beta,\lambda}^n (x_n,y_n) = \bigg( \lambda^n x_n + (1 - \lambda) \sum_{k = 0}^{n - 1} a_{- k} \lambda^k, \ y \bigg), \]
and
\[
 T_{\beta,\lambda}^n (x,y) = (x_{-n}, y_{-n}) = \biggl( x_{-n}, \ \beta^n y - \sum_{k=1}^{n} \beta^{n-k} a_k \biggr).
\]
Hence
\begin{align*}
x &= \lambda^n x_n + (1 - \lambda) \sum_{k = 0}^{n - 1} a_{- k} \lambda^k, \\
y &= \beta^{-n} y_{-n}+ \sum_{k = 1}^n \beta^{-k} a_k.
\end{align*}

\noindent Letting $n \to \infty$ we get that all points $(x,y) \in
\bigcap_{n=0}^\infty T_{\beta, \lambda}^n (Q)$ are of the form $(\pi_1
(\boldsymbol{a}, \lambda), \pi_2 (\boldsymbol{a}, \beta))$.

For any point $(x,y) \in \Lambda$, there is sequence $(x^{(k)},
y^{(k)})$ of points from $\bigcap_{n=0}^\infty T_{\beta, \lambda}^n (Q)$
that converges to $(x,y)$. But each of the points $(x^{(k)}, y^{(k)})$ is
of the form $(\pi_1 (\boldsymbol{a}^{(k)}, \lambda), \pi_2
(\boldsymbol{a}^{(k)}, \beta))$ for some $\boldsymbol{a}^{(k)}\in
S_\beta$ . Since the space $S_\beta$ is closed we conclude that $(x,y)
\in \Lambda$ is also of this form.

On the other hand, $T_{\beta,\lambda} \bigl( \pi_1 (\boldsymbol{a},
\lambda), \pi_2 (\boldsymbol{a}, \beta) \bigr) = (\pi_1 (\sigma
\boldsymbol{a}, \lambda), \pi_2 (\sigma \boldsymbol{a}, \beta))$, so
the set of points of the form $(\pi_1 (\boldsymbol{a}, \lambda), \pi_2
(\boldsymbol{a}, \beta)$ is contained in $\Lambda$. This proves
\eqref{5eq:lambda}.

We are now going to describe the unstable manifolds using the symbolic
representation. Let
\begin{equation}\label{5eq:projection}
  \pi (\boldsymbol{a}, \beta, \lambda) = (\pi_1 (\boldsymbol{a},
  \lambda), \pi_2 (\boldsymbol{a}, \beta)).
\end{equation}
Consider a sequence $\boldsymbol{a} \in S_\beta$ and the corresponding
point $p = \pi (\boldsymbol{a}, \beta, \lambda)$. In the symbolic
space, $T_{\beta,\lambda}$ acts as the left-shift, so the local unstable manifold of $p$ corresponds to the set of sequences $\boldsymbol{b}$ such that $a_k = b_k$ for $k \leq 0$.

For $\lambda \leq 1 / 2$, $\pi$ is injective on $S_\beta$ so the local
unstable manifold of $p$ is unique. If $\lambda > 1/2$, then $\pi$
need not be injective on $S_\beta$, so the local unstable manifold of
$p$ need not be unique. Indeed, when $\pi$ is not injective there are $\boldsymbol{a} \neq
\boldsymbol{b}$ such that $p = \pi (\boldsymbol{a}, \beta, \lambda) =
\pi ( \boldsymbol{b}, \beta, \lambda)$, giving rise to different unstable manifolds.

Because of the description \eqref{5eq:Sbeta+} we have that $\pi
(\boldsymbol{b}, \beta, \lambda)$ is in the unstable manifold of $\pi
(\boldsymbol{a}, \beta, \lambda)$ if $(b_1, b_2, \ldots) \leq (a_1,
a_2, \ldots)$. Hence for the unstable manifold of $\pi
(\boldsymbol{a}, \beta, \lambda)$, there is a maximal
$\boldsymbol{c}$, with $c_k = a_k$ for all $k \leq 0$, such that $\pi
(\boldsymbol{c}, \beta, \lambda)$ is contained in the unstable
manifold. For this $\boldsymbol{c}$ we have that the unstable manifold
is the set
\[
  \{\, (x,y) : x = \pi_1 (\boldsymbol{a}, \lambda), y \leq \pi_2
  (\boldsymbol{c}, \beta) \, \},
\]
\mbox{i.e.} a vertical line. So, if $\boldsymbol{a}$ is such that
$(a_1,a_2 \dots)$ does not end with a sequence of zeros, then the
unstable manifold has positive length. Since $\Lambda$ is a union of
unstable manifolds, we conclude that $\Lambda$ is the union of
line-segments of the form $\{\, (x,y) : x \text{ fixed, } 0 \leq y \leq
c \,\}$.

We will be using the symbolic representation of $\Lambda$ given by
\eqref{5eq:lambda}, so we transfer the measure $\mu_\textsc{srb}$ to a
measure $\eta$ on $S_\beta$ by $\eta = \mu_\textsc{srb} \circ \pi
(\cdot, \beta, \lambda)$. We take a closer look at this measure $\eta$
before we start the proofs. Recall, from Section \ref{5sec:betashifts},
the probability measure $\mu_\beta$ on $[0,1]$ that is invariant under
$f_\beta$ and equivalent to Lebesgue measure. We get a shift-invariant
measure on $S_\beta^+$ by taking $\mu_\beta \circ \phi_\beta$ and it
can be extended in the natural way to a shift-invariant measure
$\eta_\beta$ on $S_\beta$.

Since $\mu_\textsc{srb}$ and $\mu_\beta$ are the unique
\textsc{srb}-measures for $T_{\beta,\lambda}$ and $f_\beta$
respectively, we conclude that $\mu_\beta$ is the projection of
$\mu_\textsc{srb}$ to the second coordinate. Thus $\eta$ and
$\eta_\beta$ coincide on sets of the form $\{\, \boldsymbol{a} : a_k =
b_k, k = 1,\ldots, n \,\}$. By invariance $\eta$ and $\eta_\beta$ will
coincide. Since $\eta_\beta$ does not depend on $\lambda$ by
construction, $\eta$ does not depend on $\lambda$. We now get the
following estimates using the relation between $\eta$ and $\mu_\beta$.
\begin{align}
  \eta([a_{-n} \dots &a_0]) = \mu_\beta \Big(\phi_\beta \big(\{\,
  (x_i)_{i=1}^\infty \in S_\beta^+ : x_1 \dots x_{n+1}= a_{-n}\dots
  a_0 \,\}\big)\Big) \nonumber \\ 
  &\leq K \diameter \Big(\phi_\beta
  \big(\{\, (x_i)_{i=1}^\infty \in S_\beta^+ : x_1 \dots x_{n+1}=
  a_{-n}\dots a_0 \,\}\big)\Big) \nonumber \\ 
  & \leq K
  \beta^{-(n+1)}, \label{5eq:measureestimate}
\end{align}
where $K<\infty$ is a constant. It follows from
\eqref{5eq:measureestimate} that for $\eta$ almost all
$\boldsymbol{a}\in S_\beta$, the sequence $(a_1,a_2, \dots)$ does not
end with a sequence of zeros. As already noted, this means that the
unstable manifold is a vertical line segment of positive length. Hence
for $\eta$ almost all $\boldsymbol{a}$ the corresponding unstable
manifold is of positive length. We will use this fact in the proofs
that follow.

\begin{proof}[Proof of Theorem \ref{5the:dimension}]

Let $\beta > 1$ and pick any $\beta'\geq \beta$ such that
$\beta'\in S$. For $\eta$ almost every sequence
$\boldsymbol{a}$, the local unstable manifold of $\pi (\boldsymbol{a},
\beta, \lambda)$ corresponding to $\boldsymbol{a}$, contains a
vertical line segment of positive length. Note that this length does
not depend on $\lambda$.  Let $\omega_\delta$ be the set of sequences
$\boldsymbol{a}$, such that the corresponding local unstable manifold
of $\pi (\boldsymbol{a}, \beta, \lambda)$ has a length of at least
$\delta > 0$. Take $\delta > 0$ so that $\omega_\delta$ has positive
$\eta$-measure. Then the set $\Omega_\delta = \pi (\omega_\delta,
\beta, \lambda)$ has the same positive
$\mu_\textsc{srb}$-measure. Consider the restriction of
$\mu_\textsc{srb}$ to $\Omega_\delta$ and project this measure to $[0,
  1) \times \{ 0 \}$. Let $\mu_\textsc{srb}^\mathrm{s}$ denote this
  projection.

Take an interval $I = (c, d)$ with $0 < c < d < 1 / \beta'$. Let $t$
be a number in $(0, 1)$. We estimate the quantity
\[
  J (t) = \int_I \int_{\Omega_\delta} \int_{\Omega_\delta}
  \frac{1}{|x_1 - x_2|^t} \, \mathrm{d} \mu_\textsc{srb}^\mathrm{s}
  (x_1) \mathrm{d} \mu_\textsc{srb}^\mathrm{s} (x_2) \mathrm{d}
  \lambda.
\]
If this integral converges, then for Lebesgue almost every $\lambda
\in I$, the dimension of $\mu_\textsc{srb}^\mathrm{s}$ is at least
$t$, and so the dimension of $\mu_\textsc{srb}$ is at least $1 +
t$. Writing $J (t)$ as an integral over the symbolic space we have
that
\[
  J (t) = \int_I \int_{\omega_\delta} \int_{\omega_\delta}
  \frac{1}{|\pi_{1} (\boldsymbol{a}, \lambda) - \pi_1 (\boldsymbol{b},
    \lambda) |^t} \, \mathrm{d} \eta (\boldsymbol{a}) \mathrm{d} \eta
  (\boldsymbol{b}) \mathrm{d} \lambda.
\]
Since $\eta$ does not depend on $\lambda$ we can change order of
integration and write
\[
  J (t) = \int_{\omega_\delta} \int_{\omega_\delta} \int_I
  \frac{1}{|\pi_{1} (\boldsymbol{a}, \lambda) - \pi_1 (\boldsymbol{b},
    \lambda) |^t} \, \mathrm{d} \lambda \mathrm{d} \eta
  (\boldsymbol{a}) \mathrm{d} \eta (\boldsymbol{b}).
\]
Now, $\boldsymbol{a}, \boldsymbol{b} \in S_{\beta} \subset
S_{\beta'}$, so for $\boldsymbol{a}$ and $\boldsymbol{b}$ with $a_j =
b_j$ for $j = -k+1, \ldots, 0$ and $a_{-k} \ne b_{-k}$, we have
\[
  |\pi_1 (\boldsymbol{a}, \lambda) - \pi_1 (\boldsymbol{b}, \lambda)
 |^t = \lambda^{kt} |\pi_{1} (\sigma^{-k} \boldsymbol{a}, \lambda) -
 \pi_1 (\sigma^{-k} \boldsymbol{b}, \lambda) |^t = \lambda^{kt}
 |g(\lambda)|^t,
\]
where $g$ is of the form \eqref{5eq:powerseries}. Since $I=[c,d]
\subset [0,1/\beta']$, we can use the transversality from Lemma
\ref{5lem:transversality} to conclude that
\begin{equation}\label{5eq:innerintegral}
  \int_I \frac{d\lambda}{|\pi_1 (\boldsymbol{a}, \lambda) - \pi_1
    (\boldsymbol{b}, \lambda) |^t}  \leq c^{-kt}  \int_I \frac{d\lambda}{|g(\lambda)|^t}\leq C c^{-kt}
\end{equation}
for some constant $C$. We can write $S_\beta
\times S_\beta=A\cup B$, where

\begin{align*}
  A = & \bigcup_{k = 1}^\infty \bigcup_{[a_{-k + 1}, \ldots, a_0]} [0,
    a_{-k + 1}, \ldots, a_0] \times [1, a_{-k + 1}, \ldots, a_0]
  \\ \cup& \bigcup_{k = 1}^\infty \bigcup_{[a_{-k + 1}, \ldots, a_0]}
     [1, a_{-k + 1}, \ldots, a_0] \times [0, a_{-k + 1}, \ldots, a_0],
\end{align*}
and 
\[
  B = \bigcup_{\boldsymbol{a} \in S_\beta} \{\boldsymbol{a}\} \times
  \{\boldsymbol{a} \}.
\]
Since $\eta(\boldsymbol{a})=0$ for all $\boldsymbol{a} \in S_\beta$,
we can replace $\omega_\delta\times \omega_\delta$ by $A$ in the
estimates, so after using \eqref{5eq:innerintegral} we get
\begin{align*}
  J (t) &\leq \sum_{k = 1}^\infty \sum_{[a_{-k + 1}, \ldots, a_0]} 2 C
  c^{-kt} \int_{[0, a_{-k + 1}, \ldots, a_0]} \int_{[1, a_{-k + 1},
      \ldots, a_0]} \, \mathrm{d} \eta \mathrm{d} \eta \\ &\leq
  \sum_{k = 1}^\infty \sum_{[a_{-k + 1}, \ldots, a_0]} 2C K c^{-kt}
  \beta^{-k} \int_{[1, a_{-k + 1}, \ldots, a_0]} \, \mathrm{d} \eta
  \\ &\leq 2CK \sum_{k = 0}^\infty c^{-kt} \beta^{-k},
\end{align*}
by \eqref{5eq:measureestimate} and the fact that $\eta$ is a probability
measure. This series converges provided that $t < \frac{\log
  \beta}{\log 1/c}$.

We have now proved that for \mbox{a.e.} $\lambda$ in $I = (c, d)$, the
dimension of the \textsc{srb}-measure is at least $1 + \frac{\log
  \beta}{\log 1/c}$. To get the result of the theorem, we let
$\varepsilon > 0$ and write $I = (0, 1 / \beta')$ as a union of
intervals $I_n = (c_n, d_n)$ such that $ \frac{\log \beta}{\log 1/c_n}
> \frac{\log \beta}{\log 1/d_n} - \varepsilon$. Then the dimension is
at least $1 + \frac{\log \beta}{\log c_n} \geq 1 + \frac{\log
  \beta}{\log 1/\lambda} - \varepsilon$ for \mbox{a.e.}  $\lambda \in
I$. Since $\varepsilon$ and $\beta'$ was arbitrary this proves the
theorem.
\end{proof}

\begin{proof}[Proof of Theorem \ref{5the:abscont}]

In \cite{PeresSolomyak}, Peres and Solomyak gave a simplified proof of
Solomyak's result from \cite{Solomyak}, about the absolute continuity
of the Bernoulli convolution $\sum_{k = 1}^\infty \pm \lambda^k$. The proof that follows uses the method from \cite{PeresSolomyak} and we refer to that
paper for omitted details.

Let $\gamma \in S$, pick $\varepsilon$ according to Lemma
\ref{5lem:transversality} and let $\beta$ be such that $1/\beta\in
    [1/\gamma, 1/\gamma+\varepsilon)$. Let
      $\mu_\textsc{srb}^{\mathrm{s}}$ be the projection of
      $\mu_\textsc{srb}$ to $ [0,1] \times \{0\}$. We form
\[
  \underline D (\mu_\textsc{srb}^\mathrm{s},x)=\liminf_{r\to 0}
  \frac{\mu_\textsc{srb}^\mathrm{s}(B_r(x))}{2r},
\]
where $B_r (x) = (x-r, x+r)$, and note that
$\mu_\textsc{srb}^\mathrm{s}$ is absolutely continuous with respect to
Lebesgue measure if $\underline D
(\mu_\textsc{srb}^\mathrm{s},x)<\infty$ for
$\mu_\textsc{srb}^\mathrm{s}$ almost all $x$. Since we already have
absolute continuity in the vertical direction, it would then follow that
$\mu_\textsc{srb}$ is absolutely continuous with respect to the
two-dimensional Lebesgue measure. If
\[
  S = \int_I \int_{[0,1]} \underline D (\mu_\textsc{srb}^\mathrm{s},x)
  \mathrm{d} \mu_\textsc{srb}^\mathrm{s} (x) \mathrm{d} \lambda<\infty
  ,
\]
for an interval $I$, then $\mu_\textsc{srb}^\mathrm{s}$ is absolutely continuous
for almost all $\lambda\in I$. So if we prove that $S$ is bounded for
$I = [c, 1/\gamma + \varepsilon]$, where $c>1/\beta$ is arbitrary,
then we are done.

Let $I=[c, 1/\gamma + \varepsilon]$ for some fixed $c>1/\beta$. By
Fatou's Lemma we get
\begin{align*}
S &\leq \liminf_{r\to 0} (2r)^{-1} \int_I \int_{[0,1]}
\mu_\textsc{srb}^\mathrm{s}(B_r (x)) \, \mathrm{d}
\mu_\textsc{srb}^\mathrm{s} (x) \mathrm{d} \lambda \\ &= \liminf_{r\to
  0} (2r)^{-1} \int_I \int_{S_{\gamma}} \eta (B_r (\boldsymbol a,
\lambda)) \, \mathrm{d} \eta(\boldsymbol a) \mathrm{d} \lambda.
 \end{align*}
where $B_r (\boldsymbol{a}, \lambda) = \{\, \boldsymbol{b} : | \pi_1
(\boldsymbol{a}, \lambda) - \pi_1 (\boldsymbol{b}, \lambda) | < r
\,\}$.  We have
\[
 \eta (B_r (\boldsymbol a, \lambda)) = \int_{S_\gamma}
 \chi_{\{\boldsymbol b\in S_{\gamma} \colon | \pi_1(\boldsymbol a,
   \lambda)- \pi_1(\boldsymbol b, \lambda)|\leq r\}} ( \boldsymbol a)
 \, \mathrm{d} \eta(\boldsymbol b),
\]
where $\chi$ is the characteristic function. Since $\eta$ is
independent of $\lambda$, we can change the order of integration and
we get
\begin{align*}
  S \leq \liminf_{r\to 0} (2r)^{-1} \int_{S_\gamma} \int_{S_\gamma}
  \mu_\mathrm{Leb} \{ \lambda \in I \colon | \pi_1(\boldsymbol a,
  \lambda)- \pi_1(\boldsymbol b, \lambda)|\leq r\} \, \mathrm{d}
  \eta(\boldsymbol a)\mathrm{d} \eta(\boldsymbol b),
\end{align*}
where $\mu_{Leb}$ is the one-dimensional Lebesgue measure. Now,
$\boldsymbol{a}, \boldsymbol{b} \in S_{\gamma}$, so for
$\boldsymbol{a}$ and $\boldsymbol{b}$ with $a_j = b_j$ for $j = -k+1,
\ldots, 0$ and $a_{-k} \ne b_{-k}$, we have
\[
  |\pi_1 (\boldsymbol{a}, \lambda) - \pi_1 (\boldsymbol{b}, \lambda)
 | = \lambda^{k} |\pi_{1} (\sigma^{-k} \boldsymbol{a}, \lambda) -
 \pi_1 (\sigma^{-k} \boldsymbol{b}, \lambda) | = \lambda^{k}
 |g(\lambda)|,
\]
where $g$ is of the form \eqref{5eq:powerseries}. Since $I= [c, 1/\gamma
  +\varepsilon]$ we can use the transversality from
Lemma~\ref{5lem:transversality} and we get
\begin{multline*}
  \mu_\mathrm{Leb} \{ \lambda \in I \colon | \pi_1(\boldsymbol a,
  \lambda)- \pi_1(\boldsymbol b, \lambda)|\leq r\} \leq
  \mu_\mathrm{Leb} \{\lambda \in I : |g(\lambda)|\leq r c^{-k} \}
  \\ \leq \tilde K r c^{-k},
\end{multline*}
for some constant $\tilde K<\infty$. As in the proof of
Theorem~\ref{5the:dimension}, we can disregard the set
\[
  B = \bigcup_{\boldsymbol{a} \in S_\beta} \{\boldsymbol{a}\} \times
  \{\boldsymbol{a} \}.
\]
and after using \eqref{5eq:measureestimate} we get
\begin{align*}
  S &\leq \liminf_{r\to 0} (2r)^{-1} \sum_{k = 1}^\infty \sum_{[a_{-k
        + 1}, \ldots, a_0]} 2\tilde K r c^{-k} \int_{[0, a_{-k + 1},
      \ldots, a_0]} \int_{[1, a_{-k + 1}, \ldots, a_0]} \, \mathrm{d}
  \eta \mathrm{d} \eta \\ &\leq \sum_{k = 1}^\infty \sum_{[a_{-k + 1},
      \ldots, a_0]} \tilde K K c^{-k} \beta^{-k} \int_{[1, a_{-k + 1},
      \ldots, a_0]} \, \mathrm{d} \eta \\ &\leq \tilde K K \sum_{k =
    0}^\infty (c \beta)^{-k},
\end{align*}
which converges since $c \beta >1$. Since $c>1/\beta$ was arbitrary, we are done.
\end{proof}

\end{document}